\documentclass{article}
\usepackage{amsthm}
\usepackage{amsfonts}
\usepackage{amssymb}
\usepackage{amsmath}
\usepackage{nccmath}
\usepackage{mathtools}
\usepackage{mathrsfs}
\usepackage{setspace}
\usepackage{graphicx}
\usepackage{hyperref}
\usepackage[numbers,sort&compress]{natbib}
\usepackage{enumerate}
\usepackage{authblk}

\newtheorem{theorem}{Theorem}[section]
\newtheorem{lemma}[theorem]{Lemma}

\theoremstyle{definition}
\newtheorem{definition}[theorem]{Definition}

\newtheorem{corollary}[theorem]{Corollary}

\theoremstyle{remark}

\numberwithin{equation}{section}

\DeclareMathOperator\supp{supp}
\DeclareMathOperator\ess{esssup}
\usepackage{amsmath,amssymb,graphicx} 
\usepackage[margin=1in]{geometry} 
\usepackage{enumerate}
\usepackage{authblk}
\usepackage{placeins}
\usepackage{array}
\usepackage{academicons}
\usepackage{xcolor}
\usepackage{svg}
\usepackage[utf8]{inputenc}
\providecommand{\keywords}[1]{\textbf{\textit{Keywords:}} #1}
\providecommand{\subjclass}[1]{\textbf{\textit{MSC2020:}} #1}
\begin{document}

\nocite{*} 

\title{ Real Analytic Methods in the Formulations of some Combinatorial Inequalities}

\author{Hailu Bikila Yadeta \\ email: \href{mailto:haybik@gmail.com}{haybik@gmail.com}}
  \affil{Dilla University, College of  Natural and Computational Sciences, Department of Mathematics, Dilla, Ethiopia}
\date{\today}
\maketitle
\noindent

%

  \begin{abstract}
\noindent  In this paper, we derive some new combinatorial inequalities by applying well known real analytic results like H\"{o}lder's inequality, Young's inequality, and Minkowiski's inequality to the recursively defined sequence $f_n$ of functions

\begin{align*}
  f_0(x) & = \chi_{(-1/2, 1/2)} (x), \nonumber \\
  f_{n+1}(x) & = f_n(x+1/2)+ f_n(x-1/2), n \in \mathbb{N}\,\cup \,\{0\}.
\end{align*}

Towards this goal, we derive the closed form of the aforementioned sequence  $ (f_n)_{n\in \mathbb{N}\,\cup \,\{0\}}$  of functions and  show that it is a sequence of simple functions that are linear combinations of characteristic functions of some unit intervals $ I_{n,i},\, i=0,1, ..., n $, with values the binomial coefficients $ \binom{n}{i}$ on each unit interval $I_{n,i}$. We show that $ f_n \in L^p(\mathbb{R})),\, 1\leq p \leq \infty $. Besides applying real analytic methods to formulate some combinatorial inequalities, we also illustrate the application of some combinatorial identities. For example, we use the Vandermonde convolution (or Vandermonde identity), in the study of some properties of the sequence of functions $ (f_n)_{n\in\mathbb{ N}\cup \{0\}}$. We show how the $L^2$ norm of $f_n$ is related to the Catalan numbers.
\end{abstract}

\\
\\
\section{Introduction}
The binomial coefficients denoted by $ \binom{n}{i}$  are integers that are the numerical coefficients in the expansion for the polynomial $B_n(x)= (x+ 1)^n $, so that
$$ B_n(x)= \sum_{i=0}^{n} \binom{n}{i}x^i= (x+1)^n. $$
The binomial coefficient
\begin{equation}\label{eq:binomialdefin}
  \binom{n}{i}:=\frac{n!}{i!(n-i)!}= \frac{n(n-1)(n-2)...(n-i+1)}{i(i-1)(i-2)...1},
\end{equation}
 signify the number of ways of selecting $ i $ objects out of $ n $ without replacement. The binomial coefficient $ \binom{n}{i}:= 0 $, if $ i > n $. For historical discussion of introduction of binomial theorem and binomial coefficients one may refer Cooldige \cite{JLC}. Several results  in  combinatorial identities and inequalities are available. To mention few, a complete book of collections combinatorial identities by H. W. Gould \cite{HWG} is remarkable. Horst and Prodinger \cite{AP} present identities and inequalities involving binomial coefficients.

 In this paper, we apply some combinatorial identities in the study some properties of a recursively defined sequences of functions $ (f_n)_{n\in \mathbb{N} \{0\}}$. For example, in calculating the integral of the product $ f_n(x) f_m(x)$, we use the Vandermonde convolution (or Vandermonde identity). For $p =1$ and $ p=2 $, we use some combinatorial identities to  calculate the $ L^p $ norms of $ f_n $. In the recursive definition of the sequence, an initial function $f_0$  and an operator composed of sum  two  shift operators is applied to recursively generate the elements of the sequence.  By writing the closed form of $f_n$, we show that $f_n $ is a sequence of simple functions that are linear combinations of characteristic functions of some unit intervals $ I_{n,i},\, i=0,1, ..., n $. The values of $ f_n$  on $ I_{n,i}$ are the binomial coefficients $\binom{n}{i} $. We show this fact by setting a recurrence relation  whose solutions are binomial coefficients. Other than using an established combinatorial identities, we derive some combinatorial inequalities by using some real analytic methods. For example, we apply H\"{o}lder inequality to two elements $f_m $ and  $f_n$ of  the sequence  to derive some new combinatorial inequalities. We also apply Young inequality to the convolution $ f_m \ast f_n $ to derive another sort of combinatorial inequalities. Similar work is done by applying Minkowiski's inequality. To the best understanding of the author, these inequalities are new results. In the conclusion section, the question wether well-known integer sequence other than the binomial coefficients can be generated by a recursive definition of sequence of functions is raised.
\section{Preliminaries}
In this section, we consider recursive sequences of functions that are generated by the sum of two shift operators $E^{1/2} $  and $ E^{-1/2}$ that are defined as follows:
\begin{equation}\label{eq:leftandrightshiftoprators}
   E^ {\frac{1}{2}}u(x):= u(x + 1/2), \quad  E^ {-\frac{1}{2}} u(x):=u(x-1/2).
\end{equation}
Therefore, the sum of the two shift operators are defined as
$$\left( E^ {\frac{1}{2}} + E^ {-\frac{1}{2}}\right)u(x)= u(x + 1/2 ) + u(x-1/2).$$
\begin{definition}[\textbf{Characteristic function}]
  Let $ A \subset\mathbb{ R} $. The characteristic function of the set $ A $ is denoted by $ \chi_{A}$ is the function with values $ 1$ on $ A $ and equal to zero on the complement $\mathbb{ R} \setminus A $.
\end{definition}
Characteristic functions satisfy the following  properties
\begin{equation}\label{eq:union}
   \chi_{A \cup B}(x)=  \chi_{A}(x)  + \chi_{B}(x)- \chi_{A \cap B}(x),
\end{equation}
in particular if $ A \cap B = \emptyset $, then $\chi_{A \cup B}(x)=  \chi_{A}(x)  + \chi_{B}(x) $,
\begin{equation}\label{eq:intersection}
   \chi_{A \cap B}(x) = \chi_{A}(x) \chi_{B}(x).
\end{equation}

\begin{definition}[\textbf{simple function}]
  A simple function is a finite linear combination of characteristic functions of measurable sets. A simple function $ f: \mathbb{R}\rightarrow \mathbb{R}$, can be written in the form:
  \begin{equation}\label{eq:simplefunction}
    f(x) = \sum_{j= 1}^{n} c_{j}\chi_{J_{j}}(x),
  \end{equation}
  where $c_j \in \mathbb{R}  $ and $J_{j} \subset \mathbb{R}, j=1,2,...n $ are measurable sets.
\end{definition}
A simple function always takes finite distinct values. The representation of a simple function  in form (\ref{eq:simplefunction}) is not unique.  A simple function is nonnegative if its range is a finite subset of $[0, \infty )$. A simple function $f$ in  (\ref{eq:simplefunction}) is said to be in \emph{standard form} if the sets $ J_j, j=1,2,...,n $ are pairwise disjoint and the distinct values of $f$  are $ c_j, j=1,2,...,n $ and $0$. Any simple function can be arranged to be written in standard form.
If simple function $f$ given in (\ref{eq:simplefunction}) is in standard form, then $k$-th power of $f$ is given by
\begin{equation}\label{eq:powerofsimplefunction}
    f^k(x) = \sum_{j= 1}^{n} c^k_{j}\chi_{J_{j}}(x).
  \end{equation}
 Suppose that a function $ \phi (x)$  exhibits a convergent Taylor series expansion
  $$ \phi(x)= \sum_{k=0}^{\infty}\frac{\phi^{(k)}(0)}{k!}x^k, $$
 for all $x$ in some open interval  $I$ about $ 0$  and suppose that the range of $f$ is contained in $ I$. Then we have
 \begin{align*}
    \phi((f(x))= \sum_{k=0}^{\infty}\frac{\phi^{(k)}(0)}{k!}(f(x))^k &=  \sum_{k=0}^{\infty}\frac{\phi^{(k)}(0)}{k!}\sum_{j= 1}^{n} c^k_{j}\chi_{J_{i}}(x) \\
    &=\sum_{j= 1}^{n}\left(\sum_{k=0}^{\infty}\frac{\phi^{(k)}(0)}{k!}c^k_{j}\right)\chi_{J_{j}}(x)= \sum_{j= 1}^{n}\phi(c_j)\chi_{J_{i}}(x)
 \end{align*}

\begin{definition}[\textbf{integrals of non-negative simple functions}]
  The Lebesgue integral over $ \mathbb{R}$ of a nonnegative simple function $ f $ written in the form (\ref{eq:simplefunction})  is
  \begin{equation}\label{eq:lebegueintegralofsimple}
    \int_{\mathbb{R}}f = \int_{\mathbb{R}}f(x)dx = \sum_{j=1}^{n}c_{j}|J_{i}|,
  \end{equation}
  where $|J_{i}| $ is the measure of the set $J_{i} $.
\end{definition}
\begin{definition}[\textbf{Convolution}]\cite{VK}
  Given functions $f$ and $g$ on the real line, we say that their convolution is defined if for almost every $ t\in \mathbb{R} $ the function $ f(\tau)g(t-\tau)$ is Lebesgue integrable over $ \mathbb{R }$ as a function of the variable $ t$. In this case then the convolution of the functions $f$ and $g$ is the function $f\ast g $ defined for almost all $t\in \mathbb{R} $ by the formula
  \begin{equation}\label{eq:convolutionformula}
    \int_{\mathbb{R}}f(\tau)g(t-\tau)d\tau
  \end{equation}
\end{definition}
\section{Recursive sequences defined by sum  of two shift operators}
Consider the sequence of real functions $ (f_n)_{n\in \mathbb{N}\,\cup \,\{0\}}$
\begin{align}\label{eq:sumofshifts}
  f_0(x) & = \chi_{(-1/2, 1/2)} (x), \nonumber \\
  f_{n+1}(x) & = f_n(x+1/2)+ f_n(x-1/2)= \left( E^ {\frac{1}{2}} + E^ {-\frac{1}{2}}\right)f_n(x).
\end{align}
In the upcoming subsections, we study the sequence of functions $f_n$ defined in (\ref{eq:sumofshifts}) including its closed form, the $ L^p $ norm of $f_n$, and the integrals of the form $\int_{-\infty}^{\infty} f_n(x)f_m(x) dx $.

\subsection{Closed form of $f_n$}

 We derive the closed form expression of $f_n$ and then discuss some properties that are satisfied by $f_n$.
  Since $ f_n(x-1/2)$ is a shift of $ f_n$ half units to the right and $ f_n(x+1/2)$ is a shift of $ f_n$ half units to the left, $f_{n+1}$ is the superposition of the two shifted functions. Manual calculation of the first three elements of the sequence yields,
\begin{equation}\label{eq:fone}
  f_{1}(x)= \chi_{(-1,0)}(x ) + \chi_{(0,1)}(x),
\end{equation}
 \begin{equation}\label{eq:ftwo}
  f_{2}(x)= \chi_{(-3/2,-1/2)}(x)+2\chi_{(-1/2,1/2)}(x) + \chi_{(1/2,3/2)}(x),
\end{equation}

 \begin{equation}\label{eq:fthree}
  f_{3}(x)= \chi_{(-2,-1)}(x)+ 3\chi_{(-1,0)}(x) +  3\chi_{(0,1)}(x)+  \chi_{(1,2)}(x).
\end{equation}

 We want to write $f_n$ as linear combinations of characteristic functions of unit intervals with appropriate coefficients. We prove that the desired coefficients are in fact the binomial coefficients $ \binom{n}{i}$.

\begin{theorem}
  For appropriate $a_{n,i}, i=0,1,...,n $, the general closed form of the sequence (\ref{eq:sumofshifts}) is given as
  \begin{equation}\label{eq:generalsumofshifts}
    f_n(x)= \sum_{i=0}^{n} a_{n,i}\chi_{ \left( -\frac{n+1}{2}+i , -\frac{n+1}{2}+ i +1 \right)}(x).
  \end{equation}
  \end{theorem}
\begin{proof}
  We use induction on $n$. For $ n=1 $,
  $$ f_{1}(x)= \chi_{(-1,0)}(x ) + \chi_{(0,1)}(x) .$$
  Now suppose that the assumption in (\ref{eq:generalsumofshifts}) is true for some $ n\in \mathbb{N} $.
  Then using the operational definition of $ f_n $ given in (\ref{eq:sumofshifts})
  \begin{align*}
   f_{n+1}(x) & = f_n(x-1/2)+ f_n(x+1/2)  \\
     & = \sum_{i=0}^{n} a_{n,i}\left [\chi_{ \left( -\frac{n+1}{2}+ i- 1/2 , -\frac{n+1}{2}+ i +1/2 \right)}(x)+ \chi_{ \left( -\frac{n+1}{2}+i+1/2 , -\frac{n+1}{2}+ i +1+1/2 \right)}(x)\right] \\
     & = \sum_{i=0}^{n+1} a_{n+1,i}\chi_{ \left( -\frac{n+2}{2}+i , -\frac{n+2}{2}+ i +1 \right)}(x),
  \end{align*}
   where
 \begin{equation}\label{eq:cofficentssumdifference}
  \begin{cases}
a_{n+1,0} =  a_{n,0}= 1,\\
a_{n+1,i} =  a_{n,i} +  a_{n,i-1},\quad i=1,3,...,n, \\
a_{n+1,n+1} =  a_{n,n}= 1.
\end{cases}
\end{equation}
 The solution to the recurrence relation given by (\ref{eq:cofficentssumdifference}) are binomial coefficients defined in (\ref{eq:binomialdefin}) and satisfy Pascal's identity
\begin{equation}\label{eq:binomialsum}
   \binom{n+1}{i}= \binom{n}{i} + \binom{n}{i-1},\, 1 \leq i \leq n.
\end{equation}
Therefore, $ a_{n,i}=\binom{n}{i}$ and
\begin{equation}\label{eq:fnwithbinom}
    f_n(x)= \sum_{i=0}^{n} \binom{n}{i}\chi_{ \left( -\frac{n+1}{2}+i , -\frac{n+1}{2}+ i +1 \right)}(x).
  \end{equation}
  This proves the Theorem.
\end{proof}
The sequence $f_n$ defined according to (\ref{eq:sumofshifts}) is a sequence of simple functions which are the linear combination of characteristic functions of the $n+1$ unit intervals:
\begin{equation}\label{eq:intervalIn}
     I_{n, i}:= \left( -\frac{n+1}{2}+i , -\frac{n+1}{2}+ i +1 \right),\quad n \in \mathbb{N} \cup \{0\},\, i=0,1,2,...,n.
  \end{equation}
The values  of $f_n$  on $I_{n, i},\, i=0,1,..n $,  are $\binom{n}{i}$. These are the $n+1$ numbers on the $n+1$-th row of the Pascal's Triangle that is depicted partially in Table 1 below.
\FloatBarrier
\begin{table}[ht]
\centering
\caption{ Pascal Triangle of numbers $\binom{n}{i} $ for $0 \leq n\leq 6 $ and $ 0\leq i \leq n$.}
\begin{tabular}{>{$n=}l<{$\hspace{12pt}}*{13}{c}}
0 &&&&&&&1&&&&&&\\
1 &&&&&&1&&1&&&&&\\
2 &&&&&1&&2&&1&&&&\\
3 &&&&1&&3&&3&&1&&&\\
4 &&&1&&4&&6&&4&&1&&\\
5 &&1&&5&&10&&10&&5&&1&\\
6 &1&&6&&15&&20&&15&&6&&1
\end{tabular}
\end{table}%
\FloatBarrier

\subsection{ Properties of $ f_n $  drawn from closed-form of the sequence $f_n$ }
  In this subsection, we discuss some observable properties of the sequence $f_n$ that we will use  in the upcoming sections.
\begin{definition}
  The support $ \supp (f) $  of a function $f$ is the smallest closed set outside of which the function $f$ vanishes identically. It  is the closure of the set, $\{x |\quad f(x) \neq 0 \}$.
\end{definition}
\begin{theorem}
    For each $n \in \mathbb{N}_0 \cup \{0\}$, the support $\supp(f_n)=[-(n+1)/2, (n+1)/2]$.
  \end{theorem}

\begin{theorem}
   Each $ f_n $ is even and  non-negative valued.
\end{theorem}
\begin{proof}
   From the closed form of $ f_n$ given in (\ref{eq:fnwithbinom}) we see that the values of $f_n $ are $  \binom{n}{i} $ on $ I_{n,i}, i = 0,1,...,n $ and $ 0$ elsewhere. Hence $f_n$ is nonnegative. We prove that $f_n $ is even function by induction on $ n $. For $ n=0 $, $ f_0 $ is even. Suppose that $ f_n $ is even for some $n \in \mathbb{N}_0 \cup \{0\}$. Then
  $$f_{n+1}(-x)  = f_n(-x+1/2) + f_n(-x-1/2)= f_n(x-1/2) + f_n(x+1/2)= f_{n+1}(x) .$$
  So $ f_{n+1}$ is even. Thus $f_n$ is even for all $n \in \mathbb{N}_0 \cup \{0\}$.
  \end{proof}

 Unlike integral operators, shift operators do not increase the order of smoothness of the functions they operate on. For shift operators $E^ {\frac{1}{2}} $ and  $E^ {-\frac{1}{2}} $, the jump discontinuities are shifted half units to the left and the right respectively. The operator $\left( E^ {\frac{1}{2}} + E^ {-\frac{1}{2}}\right)$ when operated on $f_n$ to yield $f_{n+1}$, expands the support of $f_n$ by one unit and adds the number of jump discontinuities by one. For example, the jump discontinuities of $f_0 $ are the points $x_{0,0}=-1/2 $ and $x_{0,1}= 1/2 $. That of $f_1$ are $x_{1,0}=-1 $,  $x_{1,1}= 0 $, and $x_{1,2}= 0$. That of $f_2$ are $x_{2,0}= -3/2 $, $x_{2,1}= -1/2 $, $x_{2,2}= 1/2$, and $x_{2,3}= 3/2 $. We generalize the form of the jump discontinuities of $ f_n $ in the next theorem.
 \begin{theorem}[\textbf{Jump discontinuities of $f_n$}]
  For $ n\in \mathbb{N} \cup \{0\}$, the points $x_{n,i}= i-\frac{n+1}{2}, i=0, 1,2,...,n+1 $, are the points of jump discontinuities of $f_n$.
 \end{theorem}
 \begin{proof}
   We use induction on  $n$. For $n=0 $ the hypothesis holds true as the points of jump discontinuities of $f_0$ are $ x= \pm 1/2 $. Suppose that the assumption is true for some arbitrary $ n \in \mathbb{N} \cup \{0\} $. Then the jump discontinuities of $ f_{n+1}(x)= f_n(x- 1/2) + f_n(x-1/2)$ are those points $x$ such that either $x- 1/2 $ or $x + 1/2  $ are the jump discontinuities of  $f_n$. That is, the points where $x= i-\frac{n+1}{2}$ or $x = i-\frac{n-1}{2}$ for some $i \in \{0, 1,2,...,n+1\}$. These are the points $ x_{n+1,i}= i-\frac{n+2}{2}, i=0,1,...,n+2 $. These  are the jump discontinuities of $f_{n+1}$ induced from the induction hypothesis. Thus the theorem is proved.
 \end{proof}
\subsection{The powers, and the $L^p$ norms of $f_n$}
\begin{definition}[\textbf{The Lebesgue space $ L^p(\mathbb{R})$}]
  Let $1\leq p < \infty $. The function space, $ \{f|\, f : \mathbb{R} \rightarrow \mathbb{R } \}$, of measurable  functions satisfying $\int_{\mathbb{R}} |f(x)|^p dx < \infty $ is denoted by $ L^p(\mathbb{R})$.
  For $f\in L^p(\mathbb{R})$ we define the corresponding norm as
  \begin{equation}\label{eq:normdefinitionforLp}
   \|f\|_p := \left(\int_{-\infty}^{\infty}|f(x)|^p dx\right)^{\frac{1}{p}}, \quad 1 \leq p < \infty.
\end{equation}
For $p = \infty$, then $ L^\infty (\mathbb{R})$ is the set $f : \mathbb{R} \rightarrow \mathbb{R } $ of essentially bounded measurable functions. For $f \in  L^\infty (\mathbb{R})$ the norm is defined as:
$$ \|f\|_\infty :=\ess \limits_{x\in \mathbb{R}}|f(x)| $$
\end{definition}

As $f_n$ are simple functions in standard form, the $p$-power  $ f^p_n$  of $f_n$, are calculated according to (\ref{eq:powerofsimplefunction}), are also simple functions given by

\begin{equation}\label{eq:thepthpoweroffn}
   f_n^p (x)= \sum_{i=0}^{n} \binom{n}{i}^p \chi_{I_{n,i}}(x),\, 1 \leq p < \infty,\, n \in \mathbb{N} \cup\{0\}.
\end{equation}
From (\ref{eq:normdefinitionforLp}), (\ref{eq:thepthpoweroffn}), and (\ref{eq:lebegueintegralofsimple}) we have

 \begin{equation}\label{eq:Lpnormforfn}
    \|f_n\|_p  = \left( \sum_{i=0}^{n}\binom{n}{i}^p  \right)^{\frac{1}{p}}, \quad 1 \leq p < \infty, \, n \in \mathbb{N} \cup\{0\}.
\end{equation}

In particular, for $p=1$ and $ p=2$
\begin{equation}\label{eq:L1normforfn}
    \|f_n\|_1 = \int_{-\infty}^{\infty}|f_n(x)| dx =  \sum_{i=0}^{n}\binom{n}{i} = 2^n,
\end{equation}

and
\begin{equation}\label{eq:L2normforfn}
    \|f_n\|_2 = \left(\int_{-\infty}^{\infty}|f_n(x)|^2 dx\right)^{\frac{1}{2}} = \left( \sum_{i=0}^{n}\binom{n}{i}^2  \right)^{\frac{1}{2}}=\sqrt{\binom{2n}{n}}.
\end{equation}
The mean integral of $|f_n|^2$ over the support $ [-\frac{n+1}{2},\frac{n+1}{2} ]$ of $f_n$ yields
\begin{equation}\label{eq:meansquareL2normforfn}
     \frac{1}{n+1}\int_{-\frac{n+1}{2}}^{\frac{n+1}{2}}|f_n(x)|^2 dx = \frac{1}{n+1} \int_{-\infty}^{\infty}|f_n(x)|^2 dx = \frac{1}{n+1} \sum_{i=0}^{n}\binom{n}{i}^2 = \frac{1}{n+1}\binom{2n}{n},
\end{equation}
  the \emph{Catalan numbers}. The  $ L^\infty$ norm of $f_n$ is defined depending wether $n$ is even or odd. If $ n$ is even, the $ L^\infty$ norms of $f_n$ is
 \begin{equation}\label{eq:Linfinitynormeven}
   \|f_n\|_\infty = \max \limits_{0\leq i\leq n}\binom{n}{i}= \binom{n}{\frac{n}{2}},
 \end{equation}
 whereas if $n$ is odd,
 \begin{equation}\label{eq:Linfinitynormodd}
   \|f_n\|_\infty = \max \limits_{0\leq i\leq n}\binom{n}{i}= \binom{n}{\frac{n-1}{2}}= \binom{n}{\frac{n+1}{2}}.
 \end{equation}
  For $ n\in \mathbb{N }\cup \{0\}$,  $\|f_n\|_\infty$ is  the greatest numerical coefficient appearing in the expansion of $ (1+x)^n$. That is the same as the largest number that appear on the $n$-th row of the Pascal triangle that is partly  displayed in Table 1. From the above results, we conclude that $ f_n \in L^p(\mathbb{R}), \, 1 \leq p \leq \infty,\, n \in \mathbb{N }\cup \{0\} $.

 From property (\ref{eq:intersection}), for $I_{n,i} $  given in (\ref{eq:intervalIn})
 \begin{equation}\label{eq:chiproduct}
    \chi_{ I_{n,i} \cap I_{n,j }}(x)= \chi_{ I_{n,i}}(x) \chi_{ I_{n,j}}(x)=  \delta_{i,j}\chi_{ I_{n,i}}(x) ,\quad x \in \mathbb{R}, i,j\in \{ 0,1,2,...,n\},
 \end{equation}
 where $ \delta_{i,j}$ is the Kronecker delta defined as $ \delta_{i,j}= 0, \,i\neq j $ and $ \delta_{i,i}= 1$.

 As an application of real analytic method to derive the combinatorial identity, we state and derive the Vandermounde's identity which is a well known combinatorial identity.
\begin{lemma}[\textbf{Vandermonde's identity/ convolution}]
  For any nonnegative integers $r,m,n$
  \begin{equation}\label{eq:Vandermonde}
   \binom{n+m}{r} =\sum_{k=0}^{r}\binom{m}{k}\binom{n}{r-k}.
  \end{equation}
  \end{lemma}

  \begin{proof}
 First, let us evaluate the  convolution of the characteristic functions of the $i$-th unit interval in the $\supp(f_m)$ and the $j$-th unit interval in $\supp(f_n)$, $  \chi_{ I_{m,i}}(x) \ast \chi_{ I_{n,j}}(x)$.
\begin{align*}
  x \in I_{m,i} & \Leftrightarrow x \in \left( -\frac{m+1}{2}+i , -\frac{m+1}{2}+ i +1 \right)\\
   & \Leftrightarrow - \frac{1}{2} < x+\frac{m}{2}-i < \frac{1}{2}
\end{align*}
Therefore
\begin{equation}\label{eq:shiftoff0bym}
   \chi_{ I_{m,i}}(x) = f_0(x+\frac{m}{2}-i) = E ^{\frac{m}{2}-i} f_0(x).
\end{equation}

By using (\ref{eq:shiftoff0bym}), the convolution of the characteristic function of two unit intervals $I_{m,i} $ and $I_{n,j} $ is calculated as
\begin{equation}\label{convolutionImIn}
   \chi_{ I_{m,i}}(x) \ast \chi_{ I_{n,j}}(x) = E ^{\frac{m}{2}-i} f_0(x) \ast E ^{\frac{n}{2}-j} f_0(x)= E ^{\frac{m+n}{2}-i-j} \left(f_0\ast f_0\right) (x).
\end{equation}
The convolution of two elements $f_m $ and $ f_n $ of the sequence yield
\begin{align}\label{eq:EnEm}
  f_m(x)\ast f_n(x) & =\left( \sum_{i=0}^{m} \binom{m}{i}\chi_{ I_{m,i}}(x)\right) \ast  \left(\sum_{j=0}^{n} \binom{n}{j}\chi_{ I_{n,j}}(x)\right) \nonumber\\
   & =  \sum_{i=0}^{m}\sum_{j=0}^{n}\binom{m}{i}\binom{n}{j} \chi_{ I_{m,i}}(x) \ast \chi_{ I_{n,j}}(x) \nonumber \\
   & =\sum_{i=0}^{m}\sum_{j=0}^{n}  \binom{m}{i}\binom{n}{j} E ^{\frac{m+n}{2}-i-j} f_0(x)\ast f_0(x).
\end{align}
On the other hand,
\begin{align}\label{eq:Enplusm}
  f_m(x)\ast f_n(x) & =\left(E ^{\frac{1}{2}}+ E ^{-\frac{1}{2}}\right)^m f_0(x) \ast \left(E ^{\frac{1}{2}}+ E ^{-\frac{1}{2}}\right)^n f_0(x) \nonumber \\
   & = \left(E ^{\frac{1}{2}}+ E ^{-\frac{1}{2}}\right)^{m+n} f_0(x)\ast f_0(x)  \nonumber \\
   & = \sum_{r=0 }^{m+n}\binom{m+n}{r}E ^{\frac{m+n-r}{2}}E ^{\frac{-r}{2}}f_0(x)\ast f_0(x) \nonumber \\
   & =\sum_{r=0 }^{m+n}\binom{m+n}{r}E ^{\frac{m+n}{2}-r} f_0(x)\ast f_0(x).
\end{align}
Comparing the coefficients of $ E ^{\frac{m+n}{2}-r} f_0(x)\ast f_0(x) $ in (\ref{eq:EnEm}) and (\ref{eq:Enplusm}),  by setting $ i+j= r $  in (\ref{eq:EnEm}), we obtain
$$\sum_{ 0 \leq i+j=r \leq m+n }\binom{m}{i}\binom{n}{j}=  \binom{m+n}{r}. $$
This sum  when re-indexed and written yields the Lemma.
\end{proof}
In the prove the Vandermonde's identity we calculating $  f_0(x)\ast f_0(x)$ as it is. However, we may need the explicit result for latter use. We can proceed as follows.
\begin{align}\label{eq:step1}
  (f_0\ast f_0)(x) & = \int_{-\infty}^{\infty}f_0(x-y)f_0(y)dy  \nonumber \\
   &= \int_{-\frac{1}{2}}^{\frac{1}{2}} f_0(x-y)dy     \nonumber  \\
   & = \int_{x-\frac{1}{2}}^{x+\frac{1}{2}} f_0(y)dy
\end{align}
Now differentiating  the result in (\ref{eq:step1})  we get 
\begin{equation}\label{eq:step2}
  (f_0\ast f_0)'(x) = f_0(x+1/2)-f_0(x-1/2)= \chi_{(-1,0)}(x)- \chi_{(0,1)}(x) 
\end{equation}
and  from (\ref{eq:step2}) we get 
\begin{equation}\label{eq:step3}
  (f_0\ast f_0)(x)= \int_{-\infty}^{x} (f_0\ast f_0)'(s)ds = \int_{-\infty}^{x}(\chi_{(-1,0)}(s)- \chi_{(0,1)}(s))ds  \end{equation}
Integrating out we get 
\begin{equation}\label{eq:step4}
 (f_0\ast f_0)(x) =\begin{cases}
0, -\infty < x \leq -1,\\
1+x, -1 < x \leq 0, \\
1-x, 0<x<1, \\
0, 1\leq x < \infty.
\end{cases}
\end{equation}


\subsection{The evaluation of the integrals of the form $ \int_{-\infty}^{\infty} f_n(x)f_m(x)dx $}

\subsubsection{The case where $ n\geq m $ and $ n-m \equiv 0 (\text{mod 2})$}
 We use Vandermonde's identity in the proof of the two theorems that follow.
 \begin{theorem}\label{eq:theormmod0}
   Let $ n \geq m $ and $n\equiv m(\text{mod 2})$. Then we have
   $$ \int_{-\infty}^{\infty} f_n(x)f_m(x)dx = \binom{n+m}{\frac{n+m}{2}}.$$
 \end{theorem}
\begin{proof}
  Since $\supp(f_n)= \left[- \frac{n+1}{2}, \frac{n+1}{2}\right]$ has $ n+1$ unit intervals and $\supp(f_m)= \left[- \frac{m+1}{2}, \frac{m+1}{2}\right]$ has $ m+1$ unit intervals. According to an assumption in the theorem, $ (n+1)-(m+1)= n-m $ is even. The middle $ m+1$ unit intervals in the $ \supp (f_n)$ coincide with the $m+1$ unit intervals of $ \supp(f_m)$. We will observe these with the following steps.
  The first $ \frac{n-m}{2}$ unit intervals in $\supp(f_n) $ are
  $$ I_{n,0},\, I_{n,1},\,...,\,I_{n,\frac{n-m-2}{2}}   . $$
  The next $ m+1$ unit intervals in the $ \supp(f_n)$ are
   $$ I_{n,\frac{n-m}{2}},\,  I_{n,\frac{n-m+2}{2}},\,....,I_{n,\frac{n+m}{2}}  . $$
    The last  $ \frac{n-m}{2}$ unit intervals in the $ \supp(f_n)$ are
   $$ I_{n,\frac{n+m+2}{2}},\,  I_{n,\frac{n+m+4}{2}},\,....,I_{n,n}  . $$
   We have,
   $$I_{n,\frac{n-m}{2}}=\left(-\frac{n+1}{2}+\frac{n-m}{2}, -\frac{n+1}{2}+\frac{n-m}{2}+1 \right)=\left(-\frac{m+1}{2}, -\frac{m+1}{2}+1 \right)= I_{m,0} .$$
   Consequently,
   $$I_{n,\frac{n-m}{2}}= I_{m,0},\quad  I_{n,\frac{n-m+2}{2}}= I_{m,1},\quad I_{n,\frac{n-m+4}{2}}= I_{m,2},...,
    I_{n,\frac{n-m+2m}{2}}= I_{n,\frac{n+m}{2}}= I_{m,m}$$

    Therefore the middle $m+1$ unit intervals in $\supp (f_n)$ coincides with that of the unit intervals in $ \supp(f_m)$. Therefore the product $f_n f_m $ is a simple function given by
   \begin{equation}\label{eq:productfmfn}
   f_n(x)f_m(x) = \sum_{i=0}^{m} \binom{m}{i} \binom{n}{\frac{n-m+2i}{2}}\chi_{I_{m,i}}(x).
   \end{equation}

Now integrating the product $ f_n(x)f_m(x)$  that is given in (\ref{eq:productfmfn}) and using Vandermonde's identity and using the fact that $\binom{m}{i}=0, i > m $ and that $ n\geq \frac{n+m}{2} \geq m $, we get
  $$\int_{-\infty}^{\infty}f_n(x)f_m(x)dx = \sum_{i=0}^{m}\binom{n}{\frac{n-m}{2}+i}\binom{m}{i}= \sum_{i=0}^{m}\binom{n}{\frac{n+m}{2}-i}\binom{m}{i}= \binom{n+m}{\frac{n+m}{2}}.$$
\end{proof}
\subsubsection{The case where $ n > m $ and $ n-m \equiv 1 (\text{mod 2})$}
\begin{theorem}\label{eq:theoremmod1}
   Let $ n \geq m $ and $n-m\equiv 1(\text{mod 2})$. Then we have
   $$ \int_{-\infty}^{\infty} f_n(x)f_m(x)dx = \frac{1}{2}\binom{n+m+1}{\frac{n+m+1}{2}}.$$
 \end{theorem}
 \begin{proof}

 In this case, no unit interval $ I_{m,i}$ in $\supp(f_m)$ fit onto a unit interval $ I_{n,j}$ in $\supp(f_n)$. If we assume that $\left( -\frac{m+1}{2}+i, -\frac{m+1}{2}+i+1 \right)= \left( -\frac{n+1}{2}+j, -\frac{n+i}{2}+j+1 \right)$  for some $ i\in \{0,1,2,...,m\}$ and  for some $ j \in \{0,1,2,...,n\}$, then
$ n-m = 2(j-i)$. This contradicts the assumption that $ n-m \equiv  1(\text{mod 2})$. Define the half-unit intervals in $\supp(f_n)$ as
\begin{equation}\label{eq:halfunitintervalsoffn}
   J_{n,i}:= \left( -\frac{n+1}{2}+\frac{i}{2},  -\frac{n+1}{2}+\frac{i+1}{2} \right),\quad i=0,1,2,...,2n+1, n\in \mathbb{N}.
\end{equation}

We divide the $2n+2$ half unit intervals in  $\supp(f_n)$  given in (\ref{eq:halfunitintervalsoffn}) into three classes: The first $n-m$, the next $ 2m+2 $ and the last $n-m$ half-unit intervals. Note that $ (n-m) + (2m+2) + (n-m) = 2n+2$. The first $ n-m $  half-unit intervals in $\supp(f_n) $ are
  $$ J_{n,0},\, J_{n,1},\,...,\,J_{n,n-m-1}. $$
The next $ 2m+2$ half-unit intervals in the $ \supp(f_n)$ are
   $$ J_{n,n-m},\, J_{n,n-m +1},\,...,\,J_{n, n + m +1}. $$
The last  $ n-m $ half-unit intervals in the $ \supp(f_n)$ are
  $$ J_{n,n+m+2},\, J_{n,n+m +3},\,...,\,J_{n,2n+1}. $$
We claim that the middle $2m+2$ half-unit intervals of $\supp(f_n)$ coincides with the $2m+2$ half-unit intervals in $\supp(f_m)$. In fact, according to (\ref{eq:halfunitintervalsoffn}) we have

$$J_{n,n-m}= \left( -\frac{n+1}{2} +\frac{n-m}{2},  -\frac{n+1}{2}+ \frac{n-m}{2}+\frac{1}{2} \right)=  \left( -\frac{m+1}{2},  -\frac{m+1}{2}+ \frac{1}{2} \right)= J_{m,0} .$$
Consequently,

\begin{align*}
     J_{n,n-m}= J_{m,0}\subset I_{m,0}, && J_{n,n-m}&\subset I_{n,\frac{n-m+1}{2}}, \\
     J_{n,n-m+1}= J_{m,1}\subset I_{m,0},&& J_{n,n-m+1}&\subset I_{n,\frac{n-m +3}{2}}, \\
     J_{n,n-m+2}= J_{m,2}\subset I_{m,1},&& J_{n,n-m+2}&\subset I_{n,\frac{n-m +3}{2}},\\
     J_{n,n-m+3}= J_{m,3}\subset I_{m,1},&& J_{n,n-m+3}&\subset I_{n,\frac{n-m +5}{2}},\\
     ................\quad ............... && \quad ..................   \\
     J_{n,n+m}= J_{m,2m} \subset I_{m,m},&& J_{n,n+m}& \subset I_{n,\frac{n+m -1}{2}},\\
     J_{n,n+m+1}= J_{m,2m+1} \subset I_{m,m}. &&J_{n,n+m+1}& \subset I_{n,\frac{n+m +1}{2}}.
    \end{align*}
$f_m$ has values $\binom{m}{j},\, j=0, 1,2,...,m$ on the unit intervals $ I_{m,j}$, that are defined in (\ref{eq:intervalIn}). On the other hand, $ f_n$ has no uniform value on such unit intervals. For example, consider the first unit interval $ I_{m,0}= \left( -\frac{m+1}{2},  -\frac{m+1}{2}+ 1 \right) $. In this unit interval, $f_n$ has value $ \binom{n}{\frac{n-m-1}{2}}$ in the first half unit interval $\left( -\frac{m+1}{2},  -\frac{m+1}{2}+ \frac{1}{2} \right) $  and value $ \binom{n}{\frac{n-m+1}{2}} $ in the remaining half-unit interval $\left( -\frac{m+1}{2}+ \frac{1}{2}, -\frac{m+1}{2}+ 1 \right) $. Therefore the product $f_n(x)f_m(x) $ is a simple function given by
\begin{equation}\label{eq:simplefunction2}
   f_n(x)f_m(x) = \sum_{j=0}^{m} \binom{m}{j} \binom{n}{\frac{n-m+2j-1}{2}}\chi_{J_{m,2j}}(x)+ \sum_{j=0}^{m} \binom{m}{j} \binom{n}{\frac{n-m+2j+1}{2}}\chi_{J_{m,2j+1}}(x).
\end{equation}

Integrating the simple function $f_n(x)f_m(x)$ given in (\ref{eq:simplefunction2}), using Vandermonde's identity and the identity given in (\ref{eq:binomialsum}), we get
\begin{align*}
  \int_{-\infty}^{\infty}f_n(x)f_m(x)dx & = \frac{1}{2} \sum_{j=0}^{m}  \binom{m}{j} \left[ \binom{n}{\frac{n-m+2j-1}{2}}+\binom{n}{\frac{n-m+2j+1}{2}} \right] \\
   & = \frac{1}{2} \sum_{j=0}^{m}  \binom{m}{j}\binom{n+1}{\frac{n-m+2j+1}{2}}\\
   &= \frac{1}{2}\binom{n+m+1}{\frac{n+m+1}{2}}.
\end{align*}
\end{proof}

\section{ Formulation of some combinatorial inequalities by using some real analytic results}
  In this section, we apply some known results from real analysis on the sequences $f_n $  to formulate new combinatorial inequalities. Specifically we apply H\"{o}lder's inequality, Young's inequality, and Minkowiski's inequality from real analysis.
 \subsection{Application of Young's inequality in the formulation of some combinatorial inequalities}
\begin{lemma}[\textbf{H\"{o}lder's inequality}]
  For $ 1 < p < \infty $, let $f \in L^p(\mathbb{R}) $ and $g \in L^q(\mathbb{R}) $, where $\frac{1}{p}+ \frac{1}{q}= 1$. Then $fg \in L^1(\mathbb{R}) $ and $\|fg\|_1 \leq \|f\|_p \|g\|_q $.
\end{lemma}

\begin{theorem}
 For $ 1 < p < \infty $, let $\frac{1}{p}+ \frac{1}{q}= 1$, and $ 0 \leq n, m $ are integers. We have the following combinatorial inequalities:
\begin{equation}\label{eq:Holder1}
   \binom{n+m}{\frac{n+m}{2}} \leq \left(\sum_{i=0}^{m} \binom{m}{i}^p \right)^{\frac{1}{p}}  \left(\sum_{j=0}^{n}  \binom{n}{j}^q \right)^{\frac{1}{q}}, \quad n=m(\text{mod 2}),
\end{equation}

\begin{equation}\label{eq:Holder2}
   \frac{1}{2}\binom{n+m+1}{\frac{n+m+1}{2}} \leq  \left( \sum_{i=0}^{m} \binom{m}{i}^p  \right)^{\frac{1}{p}}  \left(\sum_{j=0}^{n}  \binom{n}{j}^q \right)^{\frac{1}{q}}, \quad n-m \equiv 1(\text{mod 2} )
\end{equation}
\end{theorem}
\begin{proof}
  Proof of inequalities given in (\ref{eq:Holder1}) and (\ref{eq:Holder2}) are obtained by applying Holder inequality to $f_m $ and $f_m$ and by using the results of Theorem \ref{eq:theormmod0}  and Theorem \ref{eq:theoremmod1}.
\end{proof}
\begin{corollary}
  \begin{equation}\label{eq:Schiwarz1}
   \binom{n+m}{\frac{n+m}{2}}^2  \leq  \binom{2m}{m}  \binom{2n}{n},\quad n=m(\text{mod 2}).
\end{equation}

\begin{equation}\label{eq:Schiwarz2}
   \frac{1}{4}\binom{n+m+1}{\frac{n+m+1}{2}}^2 \leq    \binom{2m}{m}  \binom{2n}{n}, \quad n-m \equiv 1(\text{mod 2})
\end{equation}
\end{corollary}
\begin{proof}
  The proofs of the inequalities in (\ref{eq:Schiwarz1}) and (\ref{eq:Schiwarz2}) follow from (\ref{eq:Holder1}) and (\ref{eq:Holder2}) by setting $ p = q = 2$ and using (\ref{eq:L2normforfn}).
\end{proof}

\subsection{Application of Young's inequality in the formulation of some combinatorial inequalities}
\begin{lemma}[\textbf{Young's inequality}]
  Assume that $1\leq p,q,r \leq \infty $ satisfies
  \begin{equation}\label{eq:pqrcondition}
    \frac{1}{r}=\frac{1}{p}+\frac{1}{q}-1.
  \end{equation}
  Let $ f\in L^p(\mathbb{R})$ and $ g\in L^q(\mathbb{R})$ then $ f \ast g \in L^r(\mathbb{R})$ and $\|f\ast g\|_r \leq \|f\|_p\|g\|_q$.
\end{lemma}
\begin{theorem}\label{eq:Yongwithmequalsnmod2}
  Let $ m, n \in \mathbb{N} \cup \{0\}$ with $ m-n\equiv 0 (\text{mod 2}) $. Then for any $1< p,q,r<\infty $ satisfying the condition in Young's inequality given (\ref{eq:pqrcondition}), we have
  $$ \left(\frac{2}{r+1}+  \sum_{j=0}^{m+n-1}\frac{\binom{m+n}{j+1}^{r+1}-\binom{m+n}{j}^{r+1}}{(r+1) [\binom{m+n}{j+1}-\binom{m+n}{j}]}\right)^{\frac{1}{r}}\leq \left(\sum_{j=0}^{m}\binom{m}{j}^q  \right)^{\frac{1}{q}}   \left(\sum_{j=0}^{n}\binom{n}{j}^p  \right)^{\frac{1}{p}}.  $$
\end{theorem}
\begin{proof}
  We apply young inequality to the convolution $ f_m\ast f_n$. For notational convenience, let $ \frac{m+n}{2}: = \mu $.  By (\ref{eq:Enplusm}) we have
$$ f_m(x) \ast f_n(x) = \sum_{j=0 }^{m+n}\binom{m+n}{j}E ^{\mu-j} f_0(x)\ast f_0(x)$$
But by (\ref{eq:step4})
$$f_0(x)\ast f_0(x)= (1+x)\chi_{[-1, 0)}(x) + (1-x)\chi_{[0, 1)(x)}$$
Then
\begin{equation}\label{eq:Emuoneplusx}
  E^{\mu-j}[(1+x)\chi_{[-1, 0)}(x)] = (x+\mu+1-j)\chi_{[-\mu-1 +j , -\mu +j )}(x),
\end{equation}

\begin{equation}\label{eq:Emuoneminusx}
E^{\mu-j}    \left[(1-x)\chi_{[0, 1)}(x)\right] = (1-x-\mu +j)\chi_{[-\mu +j , -\mu +j+1 )}(x).
\end{equation}
Therefore,
$$f_m (x) \ast f_n (x)= \sum_{j=0}^{m+n}\binom{m+n}{j}(x+\mu + 1-j)\chi_{I_j}(x)+ (1-x-\mu +j)\chi_{J_j}(x)  $$
where
\begin{align}\label{IjandJj}
  I_j & := [-\mu-1 +j , -\mu +j ), \nonumber \\
  J_j & := [-\mu +j , -\mu +j+1 ),\quad j= 0, 1,2,3...,m+n.
\end{align}
 From (\ref{IjandJj}) we note that the last $m+n$ intervals in the first collection $I_j$ of unit intervals and the first  $m+n$ intervals in the second collection $ J_j$ of unit intervals coincide. The first unit interval $ I_0$  from the first collection and the last unit interval $ J_{m+n}$ remain. This is given by
\begin{equation}\label{eq:IjandJjrelated}
  I_{j+1}= J_j,\quad j=0,1,2,...,m+n-1.
\end{equation}

\begin{align}\label{eq:fmconvfnsegregated}
 f_m (x) \ast f_n (x)&= (1+\mu+x)\chi_{I_0}(x) \nonumber\\
  & + \sum_{j=0}^{m+n-1}\left[\binom{m+n}{j}(1-x-\mu+j)+ \binom{m+n}{j+1}(x+\mu -j)\right]\chi_{J_j}(x) \nonumber\\
   & + (1-x-\mu +j)\chi_{J_{m+n}}(x).
\end{align}
Now (\ref{eq:fmconvfnsegregated}) can be rearranged and rewritten as
\begin{align}\label{eq:fmconvfnrewritten}
 f_m (x) \ast f_n (x)&= (1+\mu+x)\chi_{I_0}(x) \nonumber\\
  & + \sum_{j=0}^{m+n-1}\left[\binom{m+n}{j} + \left(\binom{m+n}{j+1}- \binom{m+n}{j}\right) (x+\mu -j)\right]\chi_{J_j}(x) \nonumber\\
   & + (1-x-\mu +j)\chi_{J_{m+n}}(x).
\end{align}
Therefore,
\begin{align}\label{eq:fmconvfntozr}
   |f_m (x) \ast f_n (x)|^r &= (1+\mu+x)^r \chi_{I_0}(x) \nonumber\\
  & + \sum_{j=0}^{m+n-1}\left[\binom{m+n}{j} + \left(\binom{m+n}{j+1}- \binom{m+n}{j}\right) (x+\mu -j)\right]^r \chi_{J_j}(x) \nonumber\\
   & + (1-x-\mu +j)^r \chi_{J_{m+n}}(x).
\end{align}
Let us integrate the two separate terms as follows.
\begin{equation}\label{eq:firstseparteintegral}
  \int_{I_0}(1+\mu+x)^r dx=\int_{-\mu-1}^{-\mu}(1+\mu+x)^rdx= \int_{0}^{1} (1-x)^r dx = \frac{1}{1+r},
\end{equation}
and
\begin{equation}\label{eq:lastseparateintegral}
 \int_{j_{m+n}}(1+\mu-x)^r dx=\int_{\mu}^{\mu+1}(1+\mu-x)^rdx = \int_{0}^{1} x^rdx = \frac{1}{1+r}.
\end{equation}
For notational convenience let
 \begin{equation}\label{eq:alphaandbeta}
    \beta_j:= \binom{m+n}{j}, \quad \alpha_j:= \binom{m+n}{j+1}-\binom{m+n}{j}.
 \end{equation}

 For each $j$ the integral of each term in the summation given in (\ref{eq:fmconvfntozr}) over the unit interval $ J_j$ is

 \begin{align}\label{eq:integralsofthemiddleterms}
     \int_ {J_j} (\beta_j + \alpha_j (x+\mu -j ))^r dx &= \int_{-\mu+j}^{-\mu+j+1}(\beta_j + \alpha_j (x+\mu -j ))^r dx \nonumber \\
    & = \frac{(\alpha_j+\beta_j)^{r+1}-\beta_j^{r+1}}{\alpha_j (r+1)}= \frac{\binom{m+n}{j+1}^{r+1}-\binom{m+n}{j}^{r+1}}{(r+1) \left[\binom{m+n}{j+1}-\binom{m+n}{j}\right]}.
 \end{align}
Then collecting the results in (\ref{eq:firstseparteintegral}), (\ref{eq:lastseparateintegral}), and (\ref{eq:integralsofthemiddleterms})
we get
\begin{equation}\label{eq:theLrnormoffmconvfn}
  \| f_m \ast f_n \|_r^r  = \frac{2}{1+r}+  \sum_{j=0}^{m+n-1}\frac{\binom{m+n}{j+1}^{r+1}-\binom{m+n}{j}^{r+1}}{(r+1) \left[\binom{m+n}{j+1}-\binom{m+n}{j}\right]}.
\end{equation}
Now by considering that $ f_n \in L^p(\mathbb{R}),\,f_m \in L^q(\mathbb{R}) $, by using definition of the $ L^p$ norm of $ f_n,\, n\in \mathbb{N } \cup \{0\}$ given in (\ref{eq:Lpnormforfn}), and the result given the in  (\ref{eq:theLrnormoffmconvfn}), the Theorem follows by applying Young's inequality.
\end{proof}

\begin{corollary}
Let $ m, n \in \mathbb{N} \cup \{0\}$ with $ m-n\equiv 0 (\text{mod 2}) $. Then for any $1< p =r < \infty $, we have
  $$ \left(\frac{2}{r+1}+  \sum_{j=0}^{m+n-1}\frac{\binom{m+n}{j+1}^{r+1}-\binom{m+n}{j}^{r+1}}{(r+1) [\binom{m+n}{j+1}-\binom{m+n}{j}]}\right)^{\frac{1}{r}} \leq  2^m \left(\sum_{j=0}^{n}\binom{n}{j}^r  \right)^{\frac{1}{r}} $$
\end{corollary}
\begin{proof}
  The proof follows from Theorem \ref{eq:Yongwithmequalsnmod2} by fixing $1< r = p < \infty $, and $q =1$.
\end{proof}

\begin{corollary}
  Let $ m, n \in \mathbb{N} \cup \{0\}$ with $ m-n\equiv 0 (\text{mod 2}) $, we have
  \begin{equation}\label{eq:Youngforrequalto1}
    1+ \frac{1}{2}\sum_{j=0}^{m+n-1}\binom{m+n+1}{j+1} \leq 2^{m+n}
  \end{equation}
\end{corollary}
\begin{proof}
  The proof follows from Theorem \ref{eq:Yongwithmequalsnmod2} by fixing $r = p= q= 1$. Indeed,
  $$ 1+ \frac{1}{2}\sum_{j=0}^{m+n-1}\binom{m+n+1}{j+1}= 1+ \frac{1}{2}\left(\sum_{j=0}^{m+n+1}\binom{m+n+1}{j}  -2\right)= \frac{1}{2}\sum_{j=0}^{m+n+1}\binom{m+n+1}{j} = 2^{m+n}, $$ showing that  equality holds for all $ m, n \in \mathbb{N} \cup \{0\}$ with $ m-n\equiv 0 (\text{mod 2})$.
\end{proof}

Note that the necessity of the inclusion of the condition that $ n=m (\text{mod 2})$ in Theorem \ref{eq:Yongwithmequalsnmod2} is that the denominator of the expression in the summation will never be zero in this case. The next theorem will treat the case where $ n-m \equiv 1 (\text{mod 2})$.
\begin{theorem}\label{eq:Yongwithmdiffrentnmod2}
  Let $ m, n \in \mathbb{N} \cup \{0\}$ with $ m-n \equiv 1 (\text{mod 2}) $. Then for any $p,q,r $ satisfying the condition in Young's inequality given by (\ref{eq:pqrcondition}), we have
  $$ \left(\frac{2}{r+1}+\binom{m+n}{\frac{m+n-1}{2}}^r + \sum_{\substack{j=0 \\ j\neq \frac{m+n-1}{2}}}^ {m+n-1}\frac{\binom{m+n}{j+1}^{r+1}-\binom{m+n}{j}^{r+1}}{(r+1) [\binom{m+n}{j+1}-\binom{m+n}{j}]}\right)^{\frac{1}{r}}\leq \left(\sum_{j=0}^{m}\binom{m}{j}^q  \right)^{\frac{1}{q}}   \left(\sum_{j=0}^{n}\binom{n}{j}^p  \right)^{\frac{1}{p}}  .$$
\end{theorem}

\begin{proof}
  For $ m-n \equiv 1 (\text{mod 2}) $, $\binom{m+n}{j+1}= \binom{m+n}{j} $ when $ j= \frac{m+n-1}{2}$. So, by the definition of $\alpha_j $ given in (\ref{eq:alphaandbeta}), $ \alpha_{\frac{m+n-1}{2}}= 0 $. The integral for the term corresponding to the index $ j= \frac{m+n-1}{2}$ is  the integral of the constant $(\beta_{\frac{m+n-1}{2}} )^r = \binom{m+n}{\frac{m+n-1}{2}}^r $ over the unit interval $ J_{\frac{m+n-1}{2}}$. This yields $\binom{m+n}{\frac{m+n-1}{2}}^r  $. The integrals of other terms  are similar to that of Theorem \ref{eq:Yongwithmequalsnmod2}. This proves the Theorem.
\end{proof}

\begin{corollary}
  Let $ m, n \in \mathbb{N} \cup \{0\}$ with $ m-n \equiv 1 (\text{mod 2}) $. Then for any $1< p =r <\infty  $ satisfying the condition in Young's inequality given by (\ref{eq:pqrcondition}), we have
  $$ \left(\frac{2}{r+1}+\binom{m+n}{\frac{m+n-1}{2}}^r + \sum_{\substack{j=0 \\ j\neq \frac{m+n-1}{2}}}^ {m+n-1}\frac{\binom{m+n}{j+1}^{r+1}-\binom{m+n}{j}^{r+1}}{(r+1) [\binom{m+n}{j+1}-\binom{m+n}{j}]}\right)^{\frac{1}{r}}\leq   2^m  \left(\sum_{j=0}^{n}\binom{n}{j}^r  \right)^{\frac{1}{r}}  .$$
\end{corollary}

\begin{proof}
  The proof follows from Theorem \ref{eq:Yongwithmdiffrentnmod2} by fixing $1< r = p < \infty $, and $q =1$.
\end{proof}

\begin{corollary}
  Let $ m, n \in \mathbb{N} \cup \{0\}$ with $ m-n\equiv 1 (\text{mod 2}) $, we have
  \begin{equation}\label{eq:Youngforrequalto1mdiffeertonmod2}
    1+  \binom{m+n}{\frac{n+m-1}{2}} + \frac{1}{2}\sum_{\substack{j=0 \\ j\neq \frac{m+n-1}{2}}}^ {m+n-1}\binom{m+n+1}{j+1} \leq 2^{m+n}
  \end{equation}
\end{corollary}
\begin{proof}
  The proof follows from Theorem \ref{eq:Yongwithmdiffrentnmod2} by setting $ r = p= q=1$.
  In fact, for  every $ m, n \in \mathbb{N} \cup \{0\}$ with $ m-n\equiv 1 (\text{mod 2})$,
   \begin{align*}
    1+  \binom{m+n}{\frac{n+m-1}{2}} + \frac{1}{2}\sum_{\substack{j=0 \\ j\neq \frac{m+n-1}{2}}}^ {m+n-1}\binom{m+n+1}{j+1} &= 1 + \binom{m+n}{\frac{n+m-1}{2}}- \binom{m+n+1}{\frac{n+m+1}{2}} + \frac{1}{2}\sum_{j=0}^ {m+n-1}\binom{m+n+1}{j+1}\\
      & = \binom{m+n}{\frac{n+m-1}{2}}- \binom{m+n+1}{\frac{n+m+1}{2}} + \frac{1}{2}\sum_{j=0}^ {m+n+1}\binom{m+n+1}{j}   \\
      & = \binom{m+n}{\frac{n+m-1}{2}}- \binom{m+n+1}{\frac{n+m+1}{2}}+ 2^{m+n}\\
      &= - \binom{m+n}{\frac{n+m-1}{2}}+ 2^{m+n} \leq  2^{m+n}.
   \end{align*}
   \end{proof}
\begin{lemma}\label{eq:lininitynormofconv}
\begin{equation}\label{eq:fnastfminfinity}
   \|f_n\ast f_m\|_\infty = \max \limits_{0\leq j \leq m+n}\binom{m+n}{j}= \|f_{n+m}\|_\infty
\end{equation}
  \end{lemma}
\begin{proof}
  According to (\ref{eq:fmconvfnrewritten}), $f_n \ast f_m $ is piecewise linear.We calculate the maximum value on each unit interval and then take the  maximum over all the unit intervals. The maximum value on each unit interval appears at either end points of interval, as linear functions have no interior critical points. Accordingly,
  \begin{align*}
     & \max \limits_{x \in I_0} (1+\mu+x) \chi_{I_0}(x)= \max \{0,1\}= 1,\\
     &  \max \limits_{x \in J_j} \beta_j + \alpha_j (1+\mu+x) \chi_{I_0}(x)=  \max  \limits_{j= 0,1,...,m+n-1} \left \{\binom{m+n}{j}, \binom{m+n}{j+1}\right\},\\
     & \max \limits_{x \in I_{m+n}} (1-x-\mu+j) \chi_{I_{m+n}}(x)= \max \{1,0\}= 1.
  \end{align*}
  Summarizing all the above results and taking overall maximum value we get
  $$\|f_n\ast f_m\|_\infty = \max \limits_{0\leq j \leq m+n}\binom{m+n}{j}. $$
  The second equality in (\ref{eq:fnastfminfinity}) follows from (\ref{eq:Linfinitynormeven}) and (\ref{eq:Linfinitynormodd}).
\end{proof}
\begin{theorem}

 $$ \max  \limits_{0 \leq j \leq m+n} \left \{\binom{m+n}{j} \right\} \leq 2^m \left[ \max  \limits_{0 \leq j \leq n} \left \{\binom{n}{j}\right\}\right]$$
\end{theorem}
\begin{proof}
  The proof follows from the Young's inequality with $ r = p=\infty $, and $q=1$ so that
   $$ \|f_m \ast f_n\|_\infty \leq \|f_n\|_\infty \|f_m\|_1 .$$
    We have applied Lemma \ref{eq:lininitynormofconv} to find $\|f_m \ast f_n\|_\infty  $ and used the results in (\ref{eq:L1normforfn}), and (\ref{eq:Linfinitynormeven}) or (\ref{eq:Linfinitynormodd}).
\end{proof}

\begin{corollary}
  If $ m $ and $n$ are both odd then
  $$ \binom{m+n}{\frac{m+n}{2}} \leq \binom{m}{\frac{m-1}{2}}2^n. $$
  If $ m $ and $n$ are both even then
  $$ \binom{m+n}{\frac{m+n}{2}} \leq \binom{m}{\frac{m}{2}}2^n. $$
\end{corollary}

\begin{corollary}
 Let $ n \in \mathbb{N} $. If $ n $ is odd then $\binom{2n}{n} \leq \binom{n}{\frac{n-1}{2}}2^n $. If $ n $ is even then $\binom{2n}{n} \leq \binom{n}{\frac{n}{2}}2^n $.
\end{corollary}

\subsection{ Application of Minkowiski's inequality for formulation of some Combinatorial inequalities}
\begin{lemma}[\textbf{Minkowski’s Inequality}]
 Let $ 1\leq p \leq \infty $. For all $f, g \in L^p( \mathbb{R} )$.Then
 \begin{equation}\label{eq:Minkowiskiinequality}
   \|f+g\|_p \leq \|f\|_p+ \|g\|_p.
 \end{equation}
\end{lemma}

\begin{theorem}
   Let $ 1 \leq p < \infty $. Let $n, m \in\mathbb{ N} \cup \{0\} $, with $ n \geq m $ and $ n-m \equiv 0 (\text{mod 2})$. Then
  \begin{equation}\label{eq:Minkowskimequalnmod2}
     \left[\sum_{i=0}^{\frac{n-m-2}{2}} 2\binom{n}{i}^p + \sum_{i=0}^{m} \left[\binom{m}{i}+\binom{n}{\frac{n-m}{2}+i}\right]^p \right]^{\frac{1}{p}} \leq \left[ \sum_{i=0}^{n}\binom{n}{i}^p \right]^{\frac{1}{p}} + \left[ \sum_{i=0}^{n}\binom{m}{i}^p \right]^{\frac{1}{p}}
  \end{equation}
\end{theorem}

\begin{proof}
For the proof, we apply Minkowiski's inequality given in (\ref{eq:Minkowiskiinequality}) to the elements $f_m$ and $f_n$ of the sequence defined in (\ref{eq:sumofshifts}) so that we have $ \|f_m +f_n\|_p \leq \|f_m\|_p +\|f_n\|_p$.
If $ n> m $ and $n-m \equiv 0 (\text{mod 2}) $, then the middle $ m+1$ unit intervals in the support of $ \supp (f_n) $ coincide with the $ m+1$ unit intervals in $ \supp (f_m) $. Therefore,

\begin{align}\label{eq:fnplusfm}
    f_n(x) + f_m(x)&= \sum_{i=0}^{\frac{n-m-2}{2}}\binom{n}{i}\chi_{I_{n,i}}(x)+\sum_{i=0}^{m} \left[\binom{m}{i}+\binom{n}{\frac{n-m}{2}+i}\right]\chi_{I_{m,i}}(x) \nonumber \\
     & + \sum_{i=0}^{\frac{n-m-2}{2}} \binom{n}{\frac{n+m+2}{2}+i} \chi_{I_{n,\frac{n+m+2}{2}+i}}(x),
  \end{align}
  Consideration of the supports of $f_m$ and $f_m$ with the given condition $n > m,$
\begin{equation}\label{eq:integrationwholeofR}
   \int_{-\infty}^{\infty} |f_n(x)+f_m(x)|^pdx = \int_{-\frac{n+1}{2}}^{-\frac{m+1}{2}} |f_n(x)|^p dx +\int_{-\frac{m+1}{2}}^{\frac{m+1}{2}} |f_n(x)+f_m(x)|^p dx + \int_{\frac{m+1}{2}}^{\frac{n+1}{2}} |f_n(x)|^p dx,   \end{equation}
and by the fact that $f_n $ is even function we have
\begin{equation}\label{eq:fromnplus1halftonplusonehaf}
  \int_{-\frac{n+1}{2}}^{-\frac{m+1}{2}} |f_n(x)|^p dx = \int_{\frac{m+1}{2}}^{\frac{n+1}{2}} |f_n(x)|^p dx  = \sum_{i=0}^{\frac{n-m-2}{2}}\binom{n}{i}^p,
\end{equation}
\begin{equation}\label{eq:integrationoversuppfm}
  \int_{-\frac{m+1}{2}}^{\frac{m+1}{2}} |f_n(x)+f_m(x)|^p dx  = \sum_{i=0}^{m} \left[ \binom{m}{i}+\binom{n}{\frac{n-m}{2}+i}\right]^p .
\end{equation}
Plugging the results in (\ref{eq:fromnplus1halftonplusonehaf}) and (\ref{eq:integrationoversuppfm}) into (\ref{eq:integrationwholeofR}) and then taking the $p$th root yields the left hand side of the inequality in (\ref{eq:Minkowskimequalnmod2}). The right hand of (\ref{eq:Minkowskimequalnmod2}) follows normally from the result in (\ref{eq:Lpnormforfn}).
\end{proof}

\begin{corollary}
\begin{equation*}
   \left( 2 +\sum_{i=0}^{m} \left[ \binom{m}{i}+\binom{m+1}{i+1}\right] ^p \right)^\frac{1}{p} \leq \left[ \sum_{i=0}^{m+2}\binom{m+2}{i}^p \right]^{\frac{1}{p}} + \left[ \sum_{i=0}^{m}\binom{m}{i}^p \right]^{\frac{1}{p}}
\end{equation*}
 \end{corollary}
\begin{proof}
  The proof the Corollary follows by setting $n= m+2$
\end{proof}
\begin{theorem}
   Let $ 1 \leq p < \infty $. Let $ m \in\mathbb{ N} \cup \{0\} $. Then
\begin{equation}\label{eq:nequalsmplus1}
  \left[ 2 +\frac{1}{2}\sum_{j=0}^{m} \left[\binom{m}{j}+ \binom{m+1}{j}\right]^p + \left[\binom{m}{j}+ \binom{m+1}{j+1}\right]^p \right]^{\frac{1}{p}}\leq \left[ \sum_{j=0}^{m}\binom{m}{j}^p \right]^{\frac{1}{p}} + \left[ \sum_{j=0}^{m+1}\binom{m+1}{j}^p \right]^{\frac{1}{p}}.
  \end{equation}
Let $ 1 \leq p < \infty $. Let $n, m \in\mathbb{ N} \cup \{0\} $, with $ n - m \geq 3 $ and $ n-m \equiv 1 (\text{mod 2})$. Then
\begin{align}\label{eq:Minkowskimnotequaltonmod2}
 &\left[ \sum_{i=0}^{\frac{n-m-3}{2}}\binom{n}{i}^p + \binom{n}{\frac{n-m-1}{2}}^p  + \frac{1}{2}\sum_{j=0}^{m} \left(\left[\binom{m}{j}+ \binom{n}{\frac{n-m+2j-1}{2}}\right]^p + \left[\binom{m}{j}+ \binom{n}{\frac{n-m+2j+1}{2}}\right]^p \right)\right]^{\frac{1}{p}} \nonumber \\
& \leq \left[ \sum_{j=0}^{m}\binom{m}{j}^p \right]^{\frac{1}{p}} + \left[ \sum_{j=0}^{n}\binom{n}{j}^p \right]^{\frac{1}{p}}
\end{align}
\end{theorem}

\begin{proof}

 For $n = m+1$ we have
\begin{align}\label{eq:fnplusfmnequalsmplus1}
   f_n(x)+ f_m(x) &=\chi_{J_{n,0}}(x)\nonumber \\
    & +\sum_{j=0}^{m} \left[\binom{m}{j}+ \binom{m+1}{j}\right] \chi_{J_{m,2j}}(x)+ \left[\binom{m}{j}+ \binom{m+1}{j+1}\right] \chi_{J_{m,2j+1}}(x)\nonumber \\
   & + \chi_{J_{n,2n+1}}(x),
\end{align}

whereas for  $n\geq m+3$ and $ n-m \equiv 1 (\text{mod 2})$ we have
  \begin{align}\label{eq:fnplusfmoddandevenn}
   f_n(x)+ f_m(x) &= \sum_{i=0}^{\frac{n-m-3}{2}} \binom{n}{i}\chi_{I_{n,i}}(x)+ \binom{n}{\frac{n-m-1}{2}}\chi_{J_{n,n-m-1}}(x)\nonumber \\
    & +\sum_{j=0}^{m} \left[\binom{m}{j}+ \binom{n}{\frac{n-m+2j-1}{2}}\right] \chi_{J_{m,2j}}(x)+  \left[\binom{m}{j}+ \binom{n}{\frac{n-m+2j+1}{2}}\right] \chi_{J_{m,2j+1}}(x)\nonumber \\
   & +\binom{n}{\frac{n+m+1}{2}}\chi_{J_{n,n+m+1}}(x)+ \sum_{i=0}^{\frac{n-m-3}{2}} \binom{n}{\frac{m+n+3}{2}+i}\chi_{I_{n,\frac{m+n+3}{2}+i}}(x),
\end{align}
 where $I_i$  and $J_j$ are as defined in (\ref{eq:intervalIn}) and (\ref{eq:halfunitintervalsoffn})  respectively. Now (\ref{eq:nequalsmplus1}) follows by applying Minkowski's inequality to (\ref{eq:fnplusfmnequalsmplus1}) and (\ref{eq:Minkowskimnotequaltonmod2}) follows by applying Minkowski's inequality to (\ref{eq:fnplusfmoddandevenn}).
\end{proof}

 \section{Conclusions and remarks}

 In this paper, we  derived  some new combinatorial inequalities  by applying some well known real analytic inequalities. Towards this goal, we calculated  the closed form of expression of some recursively-defined sequence of functions. We could have began our task from the closed form of the sequence. However the author believes that the task of calculation of the closed form expression of the sequence from  the recursive definition of the starting sequence has its own beauty and adds some value to the readers of this work. Different combinatorial inequalities may be derived by using different classes of sequences and some real analytic techniques.

 The key point in the study of recursively defined sequences similar to the ones given in (\ref{eq:sumofshifts}) is based on the choice of the pair $ (\mathcal{A},f_0) $, where $ \mathcal{A }$ an operator, and $ f_0 $ the initial function. Here in this paper, we used the pair $ (\mathcal{A},f_0(x))=\left(\left( E^ {\frac{1}{2}} + E^ {-\frac{1}{2}}\right),\chi_{(-1/2,1/2)}(x) \right) $. The author  technically  choose these operator and initial function $ f_0$  to generate a sequence  of simple function whose distinct values are binomial coefficients. Some important properties of the sequence are studied. Well known combinatorial identities were used in the study of the sequence. For example, we have used Vandermonde's identity in the evaluation of the integrals of the product $ f_nf_m$ over $ \mathbb{R}$.

 \section*{Conflict of interests}
The author declare that there is no conflict of interests regarding the publication of this paper.

\section*{Acknowledgment}
The author is thankful to the anonymous reviewers for their constructive and valuable suggestions.

\section*{Data availability}
The author assures the publishers to provide the data or sources of data that are used in this paper upon request.

\section*{Funding} This Research work is not funded by any institution or person.

\end{document}